\documentclass{amsproc}
\usepackage{xcolor}
\usepackage{amssymb}
\newtheorem{theorem}{Theorem}[section]
\newtheorem{lemma}[theorem]{Lemma}

\newtheorem{proposition}[theorem]{Proposition}
\newtheorem{corollary}[theorem]{Corollary}

\theoremstyle{definition}
\newtheorem{definition}[theorem]{Definition}

\theoremstyle{remark}
\newtheorem{remark}[theorem]{Remark}
\newtheorem{note}[theorem]{Note}
\numberwithin{equation}{section}
\usepackage{epsfig} 
\usepackage{epstopdf} 

\newcommand{\te}{\text}

\begin{document}

\title[Weaker quantization dimension results for self-similar measures]{Weaker quantization dimension results for self-similar measures}



\author{Saurabh Verma}
\address{Department of Applied Sciences, IIIT Allahabad, Prayagraj, India 211015}
\email{saurabhverma@iiita.ac.in}
\thanks{The first author is supported by the SEED Grant Project of IIIT Allahabad, India. The second author is financially supported by the Ministry of Education, India.}

\author{Shivam Dubey}
\address{Department of Applied Sciences, IIIT Allahabad, Prayagraj, India 211015}
%
\email{rss2022509@iiita.ac.in}

\date{January 1, 1994 and, in revised form, June 22, 1994.}



\begin{abstract}
In this paper, we investigate the quantization dimension of self-similar measures, particularly when the IFS does not satisfy the separation condition, but the sub-IFS at some level satisfies the separation condition. Further, we study the approximation of the space of Borel probability measures $\Omega(\mathbb{R}^m)$ with respect to the geometric mean error, i.e., the quantization dimension of order zero.

\end{abstract}

\maketitle

\section{Introduction}
An important factor in the study of fractal geometry is the fractal dimension. It measures a fractal's scaling across various magnifications to provide information on geometric complexity; it translates \textit{irregularity} into an exact numerical quantity, the \textit{fractal dimension}, that interpolates between traditional geometric dimensions. Since the irregularity exhibited by the fractal is not always homogeneous, a single notion of the fractal dimension is not sufficient. There are various dimensions that have been introduced in the literature to capture the different aspects of the complexity, including the Hausdorff dimension $\dim_\mathcal{H}(\cdot),$ Minkowski–Bouligand dimension (box-counting dimension) $\dim_B(\cdot),$ packing dimension $\dim_P(\cdot),$ Assouad dimension $\dim_A(\cdot),$ lower Assouad dimension (lower dimension) $\dim_L(\cdot),$ quantization dimension $\dim_r(\cdot), ~ L^q-$ dimension, information dimension, correlation dimension, entropy dimension, etc. (see \cite{Fal, Fraser, KNZ1, Mat3, Sh1, Zador, KPot}). Some dimensions, such as the Hausdorff and box dimensions, capture global scaling complexity, while others, such as the Assouad and lower dimensions, describe local extremal behaviour. In the case of fractal measures, global notions such as the information dimension and the $L^q$-dimension, as well as local dimensions such as the local dimension and the multifractal spectrum, provide insight at the local level. 
The concept of \textit{quantization} was first derived in information theory and signal processing, where it refers to the approximation of a continuous signal by a finite set of discrete code points. Zadore \cite{Zador} traces this idea and derived the asymptotic formulas for the optimal quantization error at high rates of a probability distribution. In this context, a natural question is to determine the rate of convergence of the optimal-$n$ approximation error as $n$ approaches infinity. Building on these ideas, Graf and Luschgy \cite{GL1} developed a measure-theoretic framework for quantization of probability distribution and formalized the notion of \textit{quantization coefficient} and \textit{quantization dimension}. In general, calculating the fractal dimension exactly is a tedious task. It becomes easier when the cylinders have no or minimal overlap. There are various separation properties in the literature, such as \textit{strong separation condition (SSC)}, \textit{open set condition (OSC)}, \textit{strong open set condition (SOSC)}, and \textit{exponential separation condition (ESC)}.

\section{Preliminaries} \label{Sec2}
In this section, we recall some definitions and basic results for our paper. For any $F\subseteq\mathbb{R}^m,$ we denote $|F|=\sup_{x,y\in F}\|x-y\|_2$ and $\text{Card}(F)$ as diameter and cardinality of $F,$ respectively. \\
Let $\mathfrak{K}(\mathbb{R}^m)$ denote the collection of all compact non-empty subsets of $\mathbb{R}^m$ with the Euclidean norm $\|\cdot\|_2$. The Hausdorff distance $\mathfrak{H}$ between any $E,F\in\mathfrak{K}(\mathbb{R}^m)$ is elaborated as 
$$\mathfrak{H}_d(E,F):=\inf\{\epsilon>0: E\subset F_\epsilon \text{ and } F\subset E_\epsilon\},$$ where $E_\epsilon$ denotes the $\epsilon$-neighbourhood of $E.$
\begin{definition}\cite{GL1}
    Let $\mu$ be a Borel probability measure on $\mathbb{R}^m,~r \in (0, +\infty)$ and $n \in \mathbb{N}.$  Then, the $n$th quantization error of order $r$ of $\mu$ is stated as:
    \[V_{n, r}(\mu):=\text{inf} \Big\{\int d(x, A)^r d\mu(x): A \subset \mathbb{R}^m, \ \text{Card}(A) \leq n\Big\},\]
    where $d(x,A)$ denotes the distance of point $x \in \mathbb{R}^m$ to the set $A$ with respect to the Euclidean norm $\|\cdot\|_2$ on $\mathbb{R}^m.$
    The quantization dimension of order $r$ of $\mu$ is defined by:
    $$ D_r(\mu):= \lim_{n\to \infty} \frac{r \log n}{- \log(V_{n,r}(\mu))}.$$
    If the limit exists, otherwise, we define the lower and upper quantisation dimensions of order $r$ of $\mu$ by taking the limit inferior and limit superior of the sequence, respectively.
\end{definition}
\begin{definition}\cite{GL1}
    For a given $s>0,$ the $s$-dimensional lower and upper quantization coefficient of order $r$ for $\mu$ is defined by:
    $$\liminf_{n \to \infty} n^{\frac{r}{s}}V_{n,r}(\mu),~ \text{and}~ \limsup_{n \to \infty}  n^{\frac{r}{s}}V_{n,r}(\mu),$$
    respectively.
\end{definition}
Let $\Omega(\mathbb{R}^m)$ denote the set of all Borel probability measures on the Euclidean space $(\mathbb{R}^m,\|.\|_2)$. Then,
\begin{equation*}
\begin{aligned}
d_L (\mathfrak{\mu}, \mathfrak{\nu}) :=\sup_{h\in\te{Lip}_1(\mathbb{R}^m)} \Big\{\Big|\int_{\mathbb{R}^m} h d\mu -\int_{\mathbb{R}^m} hd\mathfrak{\nu}\Big|\Big\}, \ \ \quad \mu, \mathfrak{\nu} \in \Omega(\mathbb{R}^m) , \end{aligned}
\end{equation*}
defines a metric on $\Omega(\mathbb{R}^m)$, where $\te{Lip}_1(\mathbb{R}^m)=\{h:\mathbb{R}^m\rightarrow \mathbb{R}: h\te{ is Lipschitz function}\\ \te{with Lipschitz constant }\le1\}$ (see \cite{Mat,Parth}). If $(X,d)$ is a compact metric space, then
$(\Omega(X), d_L)$ is a compact metric space (see \cite[Theorem~5.1]{B}), also note that if $X$ is separable and complete then the space $\omega(\mathbb{R}^m)$ is complete.
\par



\begin{definition}
Let $\mu, \mathfrak{\nu} \in \Omega(\mathbb{R}^m).$ The convolution of $\mu$ and $\mathfrak{\nu}$ is denoted by $\mu *\mathfrak{\nu}$ and defined by the push-forward of the product measure $\mu \times \mathfrak{\nu}$ under the mapping $(x,y)\to x+y.$ That is, for any $E \subset \mathbb{R}^m$, we have $$\mu * \mathfrak{\nu} (E)= (\mu \times \mathfrak{\nu} )\big(\{(x,y) \in \mathbb{R}^m\times \mathbb{R}^m: x+y \in E\}\big).$$
\end{definition}
\begin{definition}
Let $\mu \in \Omega(\mathbb{R}^m)$ and $x \in \mathbb{R}^m.$ The translation of $\mu$ by $x$ is defined as 
\[
(\mu+x)(E)=\mu(E+x),~~~E \subset \mathbb{R}^m.
\]
\end{definition}

Let us define
\[
\mathcal{L}(\mathbb{R}^m)=\Big\{\mu 
\in \Omega(\mathbb{R}^m): \mu= \sum_{i=1}^k a_i \delta_{x_i},~x_i \in \mathbb{R}^m, ~a_i > 0, \sum_{i=1}^k a_i=1,~~ k \in \mathbb{N} \Big\}.
\]
Here $\delta_x$ denotes the Dirac measure supported at $x \in \mathbb{R}^m.$ It is observed that $\mathcal{L}(\mathbb{R}^m)$ is dense in $\Omega(\mathbb{R}^m)$ w.r.t. $d_L$ on $\Omega(\mathbb{R}^m),$ see \cite{Bill}. The existence of an invariant measure supported on the attractor of the associated iterated function system is shown in \cite{H}. Feng et al. \cite{Feng} studied the convolutions of equicontractive self-similar measures on $\mathbb{R}.$ 
We define essential supremum norm $\|.\|_{\mathcal{L}^{\infty}}$ of a measurable function $g$ with respect to a measure $\mu$ by 
\[
 \|g\|_{\mathcal{L}^{\infty}}= \inf \{C\ge 0 : |f(x)| \le C ~\text{for $\mu$-almost every $x$} \}.
\]
Let $\mu ,\lambda \in \Omega( \mathbb{R}^m)$ be such that $\mu$ is absolutely continuous with respect to $\lambda.$ Then by Radon-Nikodym theorem, there exists a measurable function $g: \mathbb{R}^m \to [0, \infty) $ such that $$ \mu(A)= \int_A g d\lambda,~~~\text{for any Borel set}~A \subset \mathbb{R}^m.$$
Further, the function $g$ is called the Radon-Nikodym derivative or density of $\mu$ with respect to $\lambda$, and is denoted by $\frac{ d\mu}{ d\lambda}.$
We call $\mu$ has $\mathcal{L}^{\infty}$-density if it is absolutely continuous with respect to $\lambda$, and density function $g$ satisfying
 $\|g\|_{\mathcal{L}^{\infty}} < \infty.$  
\par
We are now ready to prove the upcoming result. Here and throughout the paper, we assume that $D_r(\mathfrak)$ exists whenever it occurs.

Let $\mathcal{I}:=(\{f_i\}_{i=1}^N, \{\rho_i\}_{i=1}^N)$ be a WIFS and $\mu$ is the corresponding self-similar measure supported on self-similar set $A.$
We denote set of all infinite sequences in $N$-symbols $\Lambda:=\{1,2,\ldots,N\}$ by $\Lambda^\infty$, and  $$\Lambda^n:= \{\xi= \xi_1 \xi_2\ldots \xi_n | \xi_i = 1,2,\ldots ,N\},~~ \Lambda^*= \bigcup_{n \in \mathbb{N}} \Lambda^n.$$
Now we define a sub-IFS at some $n$th level. Let $\xi \in \Lambda^*$ then $\mathcal{I}_n:= \{f_{\eta \xi}: \eta \in {\Upsilon}^n\}$ be a sub-IFS of the IFS $\mathcal{I}$ at $n$th level, where ${\Upsilon}^n \subseteq \Lambda^n.$ Further, the sub-IFS $\mathcal{I}_n$ with the probability vector $\big\{\frac{\rho_{\eta \xi}}{\sum_{\eta \in {\Upsilon}^n}} \rho_{\eta \xi} : \eta \in {\Upsilon}^n \big\}$ is called the sub-WIFS of the WIFS $\mathcal{I}$ at $n$th level.
\begin{note}
    Throughout the paper, we will consider $\Upsilon^n \subsetneq \Lambda^n$ in the construction of sub-IFS at the $n$-th level; otherwise, the self-similar measure $\mu_n$ associated with the sub-IFS will coincide with the self-similar measure $\mu$. 
\end{note}
\section{Main Results}

\begin{theorem}\label{Lem24}
    Let $\mu$ be a self-similar measure associated with a WIFS $(\{f_i\}_{i=1}^N; \\ \rho_1 ,\rho_2,\ldots,\rho_N)$ and $\mu_n$ be the self-similar measure associated with sub-WIFS $\mathcal{I}_n$ at $n$-th level. Then, $\mu_n$ is absolutely continuous with respect to $\mu$ i.e., $\mu_n << \mu.$
\end{theorem}
\begin{proof}
    Recall that, for a given probability vector $\{\varsigma_{\eta}= \frac{\rho_{ \eta \xi}}{ \sum_{\eta \in \Upsilon^n} \rho_{\eta \xi}}: \eta \in \Upsilon^n\},$ the map $\mathfrak{P}: \Omega(\mathbb{R}^m) \to \Omega(\mathbb{R}^m)$ defined by $\mathfrak{P}(\lambda)= \sum_{\eta \in \Upsilon^n} \varsigma_{\eta} \lambda \circ f_{\eta}^{-1}$ is contraction mapping with respect to the Monge-Kantorovich metric and by Banach contraction theorem it has a unique fixed point $\mu_n$. Also, independent of the choice of $\lambda \in \Omega(\mathbb{R}^m)$, the sequence $\mathfrak{P}^k(\lambda)$ converges to $\mu_n$ as $k$ approaches to infinity. Now, in order to prove that $\mu << \mu,$ first we show that if $\mu(A)=0$ for Borel set $A \subset \mathbb{R}^m,$ then $\mathfrak{P}^k(\mu)(A)= 0$ for all $k \in \mathbb{N}.$\\
    Consider, 
\begin{equation}\label{eq24}
\begin{aligned}
(\mathfrak{P} \circ \mathfrak{P})(\mu) = \mathfrak{P}\bigg(\sum_{\eta \in \Upsilon^n} \varsigma_{\eta} \mu \circ f_{\eta}^{-1} \bigg) = \sum_{\eta^1, \eta^2 \in \Upsilon^n}\varsigma_{\eta^1 \eta^2} \mu \circ f_{\eta^2 \eta^1}^{-1} 
\end{aligned}
\end{equation}
    Similarly, for $k \in \mathbb{N},$ we have 
    \[  \mathfrak{P}^k(\mu)= \sum_{\eta^1,\eta^2,\ldots,\eta^k \in \Upsilon^n} \varsigma_{\eta^1\eta^2\ldots\eta^k} \mu \circ f_{\eta^k \eta^{k-1}\ldots \eta^1}^{-1}.\]
    Since $\mu$ satisfies the self-similar equation $\mu= \sum_{i \in \Lambda} \rho_i \mu \circ f_{i}^{-1},$ we have 
    \[\mu= \sum_{i,j \in \Lambda} \rho_{ij} \mu \circ f_{ji}^{-1}.\]
    In $\zeta \in \Lambda^*$ is a word of length $l \in \mathbb{N}$ in the sub-IFS, then repeating the above process $l+n+k$ times we have 
    \begin{equation}\label{eq25}
    \begin{aligned}
        \mu(A) = \sum_{i^1,i^2,\ldots,i^{l+n+k} \in \Lambda} \rho_{i^1i^2\ldots i^{l+n+k}} \mu \circ f_{i^1i^2\ldots i^{l+n+k}}^{-1}(A)=0.
         \end{aligned}
    \end{equation}
    \end{proof}  
  A direct comparison of Equations \ref{eq24} and \ref{eq25} shows that Equation \ref{eq24} can be derived from Equation \ref{eq25} by omitting certain terms in its summation. Also, utilising the fact that the composition of similarity mapping is again a similarity mapping and $\mu$ is a Borel probability measure, we are able to establish our claim.
  Next, we claim that $\mathfrak{P}^k(\mu)$ converges setwise to $\mu_n$. Setting $\lambda = \mu$ and following the notation in \cite[Theorem 1]{XYZL}, the sequence $\{\mathfrak{P}^k(\mu)\}_{k \in \mathbb{N}}$ satisfies its hypotheses. Hence, $\mathfrak{P}^k(\mu)(A) \to \mu_n(A)$ for all Borel sets $A$, completing the proof.
\begin{corollary}\label{cor24}
    Let $\mu$ be a self-similar measure associated with a WIFS $(\{f_i\}_{i=1}^N\\; \rho_1,\rho_2,\ldots,\rho_N)$ and $\mu_n$ be the sub-IFS at $n$-th level. Then, $\mu_n$ is absolutely continuous with respect to $\mu,$ i.e., $\mu_n << \mu$ with an $L_\infty$ density.
\end{corollary}
\begin{proof}
    In order to prove the absolute continuity of measure $\mu_n$ with respect to measure $\mu$ with an $L_\infty$ density, it is sufficient to show that $\mu_n(A) \le M \mu(A)$ for some fixed constant $M>0$ and for all Borel measurable sets $A$. By Lemma \ref{Lem24} the sequence $\mathfrak{P}^k(\mu)$ converges setwise to $\mu_n$ and $\mathfrak{P}^k(\mu) \le \mu$ for all $k \in \mathbb{N}$. This implies that 
    $$\mu_n(A)= \lim_{k \to \infty} \mathfrak{P}^k(\mu(A)) \le \mu(A).$$
    This completes the assertion.
\end{proof}

\begin{proposition}
 Let $\mu$ be a self-similar measure associated with an IFS and $\mu_n$ be the self-similar measure corresponding to a sub-IFS at the $n$-th level. Then, $D_r(\mu_n) \le D_r(\mu)$ and $\dim_L(\mu_n) \le \dim_L(\mu).$
\end{proposition}
\begin{proof}
In Theorem \ref{Lem24} and Corollary \ref{cor24}, we have shown that $\mu_n$ is absolutely continuous with respect to $\mu$ with an $L_{\infty}$ density. Hence, by \cite[Theorem 3.4]{VAD}, we obtain $D_r(\mu_n) \le D_r(\mu).$

\end{proof}

\begin{remark}
We have some IFSs that does not satisfy the separation properties but the sub-IFS at some level satisfy the separation properties. Let $\mathcal{I}=\{\mathbb{R}:f_1,f_2,f_3\}$ be an IFS, where $f_1(x)=\frac{x}{4},~f_2(x)= \frac{x}{4}+t,~f_3(x)=\frac{x}{4}+\frac{3}{4}$ for a fixed constant
$ \frac{1}{16} < t < \frac{4}{16} ~\text{and}~\frac{8}{16} < t < \frac{11}{16}.$

Then $\{f_{11},f_{21},f_{31}\}$ will satisfy the SSC:
\[
f_{11}([0,1])= \Big[0,\frac{1}{16}\Big];~ f_{21}([0,1])=\Big[t,t+\frac{1}{16}\Big];~ f_{31}([0,1])=\Big[\frac{3}{4},\frac{13}{16}\Big].
\]
But the IFS $\mathcal{I}$ does not satisfy the OSC (equivalently, the SOSC) (see Hochman \cite{MHOCHMAN}).
Let $\mathcal{I}=\{\mathbb{R}:f_1,f_2,f_3\}$ be an IFS, where
\[
f_1(x)=\frac{x}{2},~f_2(x)= \frac{x}{4},~f_3(x)=\frac{x}{2}+\frac{1}{2}.
\]
Then $\dim_H(A)=1=\min\{s,1\}$ but for each $\eta \in \Lambda^*$ the IFS $\{f_{\xi \eta}: \xi \in \Lambda\}$ does not satisfy the SSC.
\end{remark}

We note the following fact, which will be used in the proof of the upcoming theorems.
\begin{note}\label{mn1}
The function $\frac{x}{r+x}$ is strictly increasing function on $(0,\infty),$ for any $r>0.$ Also, it is an easy observation that the function $ x \to \sum_{i=1}^N (\rho_i s_i^r)^x$ is strictly decreasing continuous function on $(0,\infty).$
\end{note}
\begin{theorem}\label{mqd1}
Let $\{\mathbb{R}^m; f_{i},\rho_i: i \in \Lambda\}$ be a WIFS consisting of similitude with the similarity ratio $\{s_i\}_{i=1}^N.$ For some $\xi \in \Lambda^*,$ consider the sub-IFS $:= \{f_{ \eta \xi}: \eta \in \Upsilon^n \},$ satisfying the OSC for each $n \in \mathbb{N}$. Then for any $r>0,$ the quantization dimension of order $r$ of the self-similar measure $\mu = \sum_{i \in \Lambda} \rho_{i} \mu_{\rho} \circ f_{i}^{-1}$ associated with the IFS $\{f_{i},\rho_i\}_{i=1}^N$ is $\min\{\kappa_r,m\},$ where $\kappa_r$ is uniquely determined by the expression $\sum_{i \in \Lambda} (\rho_{i} s_{i}^r)^{\frac{\kappa_r}{r + \kappa_r}}= 1.$
\end{theorem}
\begin{proof}
Given that the sub-IFS $\{f_{ \eta \xi}: \eta \in \Upsilon^n\}$ satisfies the OSC, it follows from Graf-Luschgy \cite{GL1} that the quantization dimension of the self-similar measure 
$\mu_n:= \sum_{\eta \in \Upsilon^n} \varsigma_{\eta} \mu_n \circ f_{ \eta \xi}^{-1},$
where $\mu_n$ is the self-similar measure associated with the sub-IFS and the corresponding probability vector 
$\{\varsigma_{\eta}= \frac{\rho_{ \eta \xi}}{ \sum_{\theta \in \Upsilon^n} \rho_{\theta \xi}} : \eta \in \Upsilon^n\}$ 
is determined by the unique solution $\kappa_{n,r}$ of the equation
$\sum_{\eta \in \Upsilon^n} (\varsigma_{\eta} s_{\xi \eta}^r)^{\frac{\kappa_{n,r}}{r + \kappa_{n,r}}}= 1.$
As $D_r(\mu_n)$ may be sufficientely small and tending to zero as $n$ goes to large, we may assume that $\kappa_{n,r}=D_r(\mu_n) \le D_r(\mu_{\rho}).$ For contradiction, we assume that $D_r(\mu_{\rho}) < \kappa_{r}$. Then, using Note \ref{mn1} and $\sum_{\eta \in \Upsilon^n} (\varsigma_{\eta} s_{\xi \eta}^r)^{\frac{\kappa_{n,r}}{r + \kappa_{n,r}}}=1,$ we have 
\begin{align*}
s_{\xi}^{-\frac{r \kappa_{n,r}}{r + \kappa_{n,r}}} 
&= \sum_{\eta \in \Upsilon^n} (\varsigma_{\eta}  s_{ \eta}^r)^{\frac{\kappa_{n,r}}{r + \kappa_{n,r}}}  \\
&\ge \bigg(\sum_{\eta \in \Upsilon^n} \rho_{ \eta}\bigg)^{\frac{-D_r(\mu)}{r+ D_r(\mu)}}\rho_{\xi}^{\frac{D_r(\mu)}{r+ D_r(\mu)}} \sum_{\eta \in \Upsilon^n} (\rho_{ \eta}  s_{ \eta}^r)^{\frac{D_r(\mu)}{r + D_r(\mu)}} \\
& = \bigg(\sum_{\eta \in \Upsilon^n} \rho_{\xi \eta}\bigg)^{\frac{-D_r(\mu)}{r+ D_r(\mu)}}\rho_{\xi}^{\frac{D_r(\mu)}{r+ D_r(\mu)}} \sum_{\eta \in \Upsilon^n} (\rho_{ \eta}  s_{ \eta}^r)^{\frac{D_r(\mu)}{r + D_r(\mu)}}(\rho_{ \eta}   s_{ \eta}^r)^{\frac{-\kappa_{r}}{r + \kappa_{r}}} (\rho_{ \eta}  s_{ \eta}^r)^{\frac{\kappa_{r}}{r + \kappa_{r}}} \\
&= \bigg(\sum_{\eta \in \Upsilon^n} \rho_{\eta}\bigg)^{\frac{-D_r(\mu)}{r+ D_r(\mu)}}\sum_{\eta \in \Upsilon^n} (\rho_{\eta}  s_{ \eta}^r)^{\frac{D_r(\mu) }{r + D_r(\mu)}- \frac{\kappa_{r}}{r + \kappa_{r}}} (\rho_{\eta}  s_{ \eta}^r)^{\frac{\kappa_{r}}{r + \kappa_{r}}}\\
& \ge  \sum_{\eta \in \Upsilon^n} (\rho_{\eta}  s_{\min}^{nr})^{\frac{D_r(\mu) }{r + D_r(\mu)}- \frac{\kappa_{r}}{r + \kappa_{r}}} (\rho_{\eta}  s_{ \eta}^r)^{\frac{\kappa_{r}}{r + \kappa_{r}}} \\
& \ge  s_{\min}^{nr. \bigg({\frac{D_r(\mu) }{r + D_r(\mu)}- \frac{\kappa_{r}}{r + \kappa_{r}}}\bigg)}  \sum_{\eta \in \Upsilon^n} \rho_{\eta}^{\frac{D_r(\mu) }{r + D_r(\mu)}- \frac{\kappa_{r}}{r + \kappa_{r}}} (\rho_{\eta}  s_{ \eta}^r)^{\frac{\kappa_{r}}{r + \kappa_{r}}} \\ 
 & \ge s_{\min}^{nr. \bigg({\frac{D_r(\mu) }{r + D_r(\mu)}- \frac{\kappa_{r}}{r + \kappa_{r}}}\bigg)}  \sum_{\eta \in \Upsilon^n} \rho_{\eta}^{\frac{D_r(\mu) }{r + D_r(\mu)}} (s_{ \eta}^r)^{\frac{\kappa_{r}}{r + \kappa_{r}}}
\\
&  \ge  s_{\min}^{nr. \bigg({\frac{D_r(\mu) }{r + D_r(\mu)}- \frac{\kappa_{r}}{r + \kappa_{r}}}\bigg)} \sum_{\eta \in \Upsilon^n}( \rho_{\eta} s_{ \eta}^r)^{\frac{\kappa_{r}}{r + \kappa_{r}}}\\ & \ge s_{\min}^{nr. \bigg({\frac{D_r(\mu) }{r + D_r(\mu)}- \frac{\kappa_{r}}{r + \kappa_{r}}}\bigg)}.
 \end{align*}
Since $D_r(\mu)- \kappa_r < 0$ as $n$ approaches infinity, leading to a contradiction, this completes the assertion.
\end{proof}
\begin{remark}
    By \cite[Corollary 11.4]{GL1}, for $r \in [1, \infty),$ the quantization dimension function is a non-decreasing function in $r$. One can observe that this result holds for $r \in [0, \infty)$ due to the fact that for $0 \le p \le q < \infty, ~~ \|x\|_{p} \le \|x\|_{q}.$
\end{remark}
\begin{remark}
    Let $\mu$ be a self similar measure associated with an WIFS $(\{f_i\}_{i=1}^N;\\ \{\rho_i\}_{i=1}^N).$ Then, by \cite[Lemma 3.5]{GL3} $$\lim_{r \to 0^+} D_r(\mu)= D_{0}(\mu).$$
  \end{remark}
  \subsection{Approximation of measures in terms of quantization dimension}
 \begin{lemma}\label{lem10}
     Let $\mu$ and $\nu$ be two finite measures. Then 
     $$D_{0}(\mu +\nu) = \max\{D_{0}(\mu) , D_{0}(\nu)\}.$$
 \end{lemma}
 \begin{proof}
     By the definition of the $n$th quantization error of a measure with respect to the geometric mean error, we have 
     \begin{equation}
      \begin{aligned}
          \log e_{n}(\mu+\nu) &= \inf \bigg\{ \int \log d(x,A) d(\mu+\nu)(x) : A \subset \mathbb{R}^d, \text{Card}(A) \le n\bigg\} \\& \ge 
\inf \bigg\{ \int \log d(x,A) d(\mu)(x)\bigg\} + \inf \bigg\{ \int \log d(x,A) d(\nu)(x)\bigg\} \\& = \log e_n(\mu) + \log e_{0}(\nu).
  \end{aligned}   
     \end{equation}
     This depicts that $D_{0}(\mu + \nu) \ge \max\{D_{0}(\mu), D_{0}(\nu)\}.$ Now, let $t > \max\{D_{0}(\mu), D_{0}(\nu)\},$ then by \cite[Proposition 4.3]{GL3}, we have 
     $$\lim_{n \to \infty}(\log n + t \log e_{n}(\mu)) = +\infty \quad \lim_{n \to \infty}(\log n + t \log e_{n}(\nu)) = +\infty.$$
     Let $k \in \mathbb{R}$ be any arbitrary number, then there exist natural numbers $N_1(\mu,k)$ and $N_2(\nu,k)$ such that 
     $$\log n  + t \log e_{n}(\mu) > k \quad \text{and} \quad \log n  + t \log e_{n}(\mu) > k \quad \forall n \ge N = \max\{N_1,N_2\}.$$
     Also, by the definition of the quantization error of a measure with respect to the geometric mean error, there exist sets $A(\mu,n,k)$ and $B(\nu,n,k)$ such that 
     \begin{equation*}
         \begin{aligned}
            &\log n  + t  \int \log d(x,A(\mu,n,k)) d(\mu)(x) > k, \quad \text{and} \\& \log n  + t \int \log d(x,B(\nu,n,k)) d(\nu)(x) > k \quad \forall n \ge N. 
         \end{aligned}
     \end{equation*}
  Since Card$(A(\mu,n,k) \le n)$ and Card$(B(\nu,n,k) \le n),$ it follows that Card$(A \cup B) \le n^2,$ and we have $$2 \log n + t \int \log d(x, A\cup B) d (\mu)(x) > k.$$   
  Now, for every $n > N,$ we have 
  \begin{equation*}
      \begin{aligned}
          2 \log n + t \int \log d(x, A \cup B) d(\mu + \nu)(x) & =  \log n + t \int \log d(x, A \cup B) d(\mu)(x) + \\& \log n + t \int \log d(x, A \cup B) d(\nu)(x) > k
      \end{aligned}
  \end{equation*}
     Since $k \in \mathbb{R}$ is arbitrary, we have $\lim_{n \to \infty}(\log n+ t \log e_{n}(\mu + \nu )) = +\infty$. It follows that $D_{0}(\mu + \nu) < t.$ Further, since $t > \max\{D_{0}(\mu), D_{0}(\nu)\}$ is arbitrary, it follows that $D_{0}(\mu + \nu) \le \max\{D_{0}(\mu), D_{0}(\nu)\}.$ This completes the assertion.
 \end{proof}
\begin{lemma}\label{lem09}
 Let $\mu$ be a finite Borel probability measure. Then for a fixed $\nu \in \mathcal{L}(\mathbb{R}^m),$ 
 $$D_{0}(\mu * \nu) = D_{0}(\mu).$$
\end{lemma}
\begin{proof}
    Following a similar approach to \cite[Lemma 3.10]{VAD}, we obtain the required result.
\end{proof}
Using the above lemmas, we will approximate the set $\Omega(\mathbb{R}^m)$ by several of its subclasses with respect to the quantization dimension.


\begin{theorem}
    The sets $\Omega_0({\mathbb{R}^m})=\{\mu \in \Omega(\mathbb{R}^m): \lim_{r \to 0^+}D_r(\mu) =D_0(\mu)= \dim_H(\mu)\}$ and $\Omega_{0>}({\mathbb{R}^m})=\{\mu \in \Omega(\mathbb{R}^m): \lim_{r \to 0^+}D_r(\mu) =D_0(\mu)>\dim_H(\mu)\}$  are dense in $\Omega(\mathbb{R}^m).$
\end{theorem}
\begin{proof}
Since By \cite[Remark 4.6]{GL3}, the set $\Omega_0({\mathbb{R}^m})$ is non-empty. Let $\mu \in \Omega(\mathbb{R}^m)$ be any Borel probability measure and $\epsilon >0,$ then by the fact that $\mathcal{L}(\mathbb{R}^m)$ is dense in $\Omega(\mathbb{R}^m),$ there exists a measure $\nu \in \mathcal{L}(\mathbb{R}^m)$ such that
$$d_L(\mu, \nu) < \frac{\epsilon}{2}.$$ Further, by Lemma \ref{lem09} and the similar technique followed in \cite[Theorem 3.16]{VAD}, we have there exists a compactly supported measure such that $d_L(\nu, \lambda) < \frac{\epsilon}{2},$ where $\lambda = \nu * \lambda_1$ for some $\lambda_1 \in \Omega_0({\mathbb{R}^m}).$ in this way, we have $$d_L(\mu, \lambda_1) \le d_L(\mu,\nu)+ d_L(\nu,\lambda) < \epsilon.$$
This proves that the set $\Omega_0({\mathbb{R}^m})$ is a dense subset of $\Omega(\mathbb{R}^m).$
Again, by \cite[Theorem 4.3]{GL3} the set $\Omega_{0>}({\mathbb{R}^m})$ is non-empty. Since the proof is similar, we omit it.

\end{proof}

We introduce the total variation metric (\cite{Parth}) on $\Omega( \mathbb{R}^m)$ as follows:
\[
\|\mu - \nu\|_{TV}=\sup\{|\mu(A)- \nu (A)|: A ~\text{is a Borel set}\}.
\]

\begin{lemma}\label{lemcon1}
            
Let $\lambda \in \Omega( \mathbb{R}^m)$ be a fixed measure and $\big(\mu_n\big)$ be a sequence of absolutely continuous (with respect to $\lambda$) measures in $\Omega( \mathbb{R}^m).$ If $(\frac{d\mu_n}{d\lambda})$ converges in $\mathcal{L}^1(\lambda)$-norm then $\big(\mu_n\big)$ converges to a measure $\mu$, and
$$
\frac{d\mu}{d\lambda}=\lim_{n \to \infty }\frac{d\mu_n}{d\lambda}.
$$
\end{lemma}
\begin{proof}
Let $g= \lim_{n \to \infty} \frac{d\mu_n}{d\lambda}.$ Then a measure $\mu$ defined by $$\mu(A)= \int_A g d\lambda, ~~ ~~~~ A \subset \mathbb{R}^n,$$
is in $ \Omega( \mathbb{R}^m).$ Further, we obtain 
\begin{equation*}
\begin{aligned}
d_L(\mu_n,\mu)& =\sup_{\te{Lip}(f)\leq 1} \Big\{\Big|\int_{\mathbb{R}^m} f d\mu_n -\int_{\mathbb{R}^m} fd\mu \Big|\Big\} \\ & = \sup_{\te{Lip}(f)\leq 1} \Big\{\Big|\int_{\mathbb{R}^m} f  \frac{d\mu_n}{d\lambda} d\lambda -\int_{\mathbb{R}^m} fgd\lambda \Big|\Big\} \\ & \le  \sup_{\te{Lip}(f)\leq 1} \Big\{\int_{\mathbb{R}^m} \Big|f  \frac{d\mu_n}{d\lambda} -fg \Big|d\lambda\Big\} \\ & \le  \sup_{\te{Lip}(f)\leq 1} \|f\|_{\infty} \Big\{\int_{\mathbb{R}^m} \Big|  \frac{d\mu_n}{d\lambda} -g \Big|d\lambda\Big\} \\ & \le   \int_{\mathbb{R}^m} \Big|  \frac{d\mu_n}{d\lambda} -g \Big|d\lambda.
\end{aligned}
\end{equation*}
Since $\frac{d\mu_n}{d\lambda} \to g$ in $\mathcal{L}^1(\lambda)$-norm, we have $d_L(\mu_n,\mu) \to 0$ as $n \to \infty,$ hence the proof.
\end{proof}
\begin{remark}
Under the hypothesis of the above theorem, we also have 
\begin{equation*}
\begin{aligned}
\|\mu_n-\mu\|_{TV}& =\sup\{|\mu_n(A)- \mu (A)|: A ~\text{is a Borel set}\}\\ & = \sup\Big\{\Big|\int_{A}   \frac{d\mu_n}{d\lambda} d\lambda -\int_{A} gd\lambda \Big| : A ~\text{is a Borel set}\Big\} \\ & \le  \sup \Big\{\int_{A} \Big|  \frac{d\mu_n}{d\lambda} -g \Big|d\lambda:  A ~\text{is a Borel set}\Big\} \\ & \le   \int_{\mathbb{R}^m} \Big|  \frac{d\mu_n}{d\lambda} -g \Big|d\lambda.
\end{aligned}
\end{equation*}
From the above, we infer that $\|\mu_n-\mu\|_{TV} \to 0$ as $n \to \infty.$
\end{remark}
\begin{theorem}\label{approxthm1}
Let  $\lambda \in \Omega( \mathbb{R}^m).$
Suppose $\mu$ is a absolutely continuous measure with respect to $\lambda$ such that $\dim_H \Big(Gr\big(\frac{d\mu}{d\lambda}\big)\Big) =\alpha$ for some $m \le \alpha \le m+1.$ Then there exists a sequence of absolutely continuous measures $(\mu_n) \in \Omega( \mathbb{R}^m)$ with the following conditions: $$\dim_H\Big(Gr\big(\frac{d\mu_n}{d\lambda}\big)\Big) =\alpha ~~\text{and}~~ \mu =\lim_{n \to \infty}\mu_n.$$  
\end{theorem}
\begin{proof}
In the light of an analogous result to Theorem \ref{lemcon1}, we may have a sequence of integrable (with respect to $\lambda$) functions $(g_n)$ such that $\dim_H\big(Gr(g_n)\big) =\alpha$ and  $g_n \to \frac{d\mu}{d\lambda}$ in $\mathcal{L}^1(\lambda)$-norm. 
Now, one defines a sequence of measures $(\mu_n)$ as follows $$\mu_n(A)=\int_A g_nd\lambda, ~~~~~~ A \subset \mathbb{R}^m.$$ Then, $\frac{d\mu_n}{d\lambda}=g_n$ and $\Big(\frac{d\mu_n}{d\lambda}\Big)$ converges to $\frac{d\mu}{d\lambda}.$ Lemma \ref{lemcon1} yields that the sequence $(\mu_n)$ converges to $\mu$ such that $\dim_H\Big(Gr\big(\frac{d\mu_n}{d\lambda}\big)\Big) =\alpha.$ Hence the assertion is proved. 
\end{proof}
\section*{Acknowledgements}
 The first author is supported by the SEED grant project of IIIT Allahabad. The second author is financially supported by the Ministry of Education at the IIIT Allahabad. Some parts of the paper have been presented in lectures delivered by the first author at the workshop ``Course on Exponential separation and $L^q$ dimension" organized by IIIT Allahabad on April 23, 2025.

\bibliographystyle{amsalpha}

\end{document}